\RequirePackage{fix-cm}
\documentclass[smallextended]{svjour3}       
\smartqed  
\usepackage{amssymb,amsmath,amsfonts,mathrsfs}
\usepackage[utf8]{inputenc}
\usepackage[english]{babel}
\begin{document}

\title{The Width and Integer Optimization on Simplices With Bounded Minors of the Constraint Matrices
}


\author{D. V. Gribanov        \and
        A. Y. Chirkov 
}


\institute{D. V. Gribanov \at Lobachevsky State University of Nizhny Novgorod, 23 Gagarina Avenue, Nizhny Novgorod, Russian Federation, 603950,\\
National Research University Higher School of Economics, 136 Rodionova, Nizhny Novgorod, Russian Federation, 603093,\\
\email{dimitry.gribanov@gmail.com}
\and
A. Y. Chirkov \at Lobachevsky State University of Nizhny Novgorod, 23 Gagarin Avenue, Nizhny Novgorod, Russian Federation, 603950,\\ \email{chir7@yandex.ru}
}

\date{Received: date / Accepted: date}

\maketitle

\begin{abstract}
In this paper, we will show that the width of simplices defined by systems of linear inequalities can be computed in polynomial time if some minors of their constraint matrices are bounded. Additionally, we present
some quasi-polynomial-time and polynomial-time algorithms to solve the integer linear optimization problem defined on simplices minus all their integer vertices assuming that some minors of the constraint matrices of the simplices are bounded.

\end{abstract}

\section{Introduction}
Let $A$ be a $m\times n$ integral matrix. Its $ij$-th element is
denoted by $A_{i\,j}$, $A_{i\,*}$ is $i$-th row of $A$, and $A_{*\,j}$ is $j$-th column of $A$.
Additionally, for subsets $I \subseteq \{1,\dots,m\}$ and $J \subseteq \{1,\dots,n\}$, $A_{I\,J}$ denotes the submatrix of $A$ that was generated by all rows with numbers in $I$ and all columns with numbers in $J$. Sometimes, we will change the symbols $I$ and $J$ to the symbol $*$ meaning that we take the set of all rows or columns, respectively. For a vector $b\in\mathbb{Z}^{n}$, by $P(A,b)$ we denote a polyhedron
$\{ x \in \mathbb{R}^{n} : A x \leq b\}$. The set of all vertices of a polyhedron $P$ is denoted by $vert(P)$.

Let $r(A)$  be the rank of an integral matrix $A$.
Let $\Delta_k(A)$ and $\delta_k(A)$ denote the greatest and the smallest absolute values of the determinants of all $k \times k$ sub-matrices of $A$, respectively.
Let $\Delta_{gcd}(k,A)$ and $\Delta_{lcm}(k,A)$ be the greatest common divisor and the least common multiple of the determinants of all $k \times k$ sub-matrices of $A$, respectively. Additionally, let $\Delta(A) = \Delta_{r(A)}(A)$, $\delta(A) = \delta_{r(A)}(A)$, $\Delta_{gcd}(A) = \Delta_{gcd}(r(A),A)$, and $\Delta_{lcm}(A) = \Delta_{lcm}(r(A),A)$.

Sometimes, we will use the symbol $\tilde O$ instead of the symbol $O$ for an estimation of the complexity. It means that we ignore the logarithmic factor in a complexity bound, i.e. $\tilde O(f(n))=\{g(n)|~\exists c>0~g(n)=O(f(n)log^c(f(n)))\}$.

\begin{definition}
For a matrix $B \in \mathbb{R}^{s \times n}$, $cone(B) = \{B t : t \in \mathbb{R}_+^{n} \}$ is the \emph{cone spanned by
columns of} $B$ and $conv(B) = \{B t : t \in
\mathbb{R}_+^{n},\, \sum_{i=1}^{n} t_i = 1  \}$ is the
\emph{convex hull spanned by columns of} $B$.
\end{definition}

\begin{definition}
For $v \in vert(P(A,b))$, let $N(v) = cone(A_{J(v)\,*}^\top)$, where $J(v) = \{j :\: A_{j\,*} v = b_j \}$.
The set $N(v)$ is called the \emph{normal cone of the vertex} $v$.
\end{definition}

\begin{definition}
The \emph{width of a convex body} $P$ is defined as
$ w(P)=\min\limits_{c \in \mathbb{Z}^n\setminus\{0\}} (\max\limits_{x \in P} c^\top x - \min\limits_{x \in P} c^\top x)$.
A vector $c$ minimizing the difference $\max\limits_{x \in P} c^\top x - \min\limits_{x \in P} c^\top x$ on $\mathbb{Z}^n\setminus\{0\}$
is called the \emph{flat direction of} $P$.
\end{definition}

We will use the classical flatness theorem of Khinchine \cite{KHI48}. Let $P$ be a convex body. Khinchine shows that if $P \cap \mathbb{Z}^n = \emptyset$, then $w(\,P\,) \leq f(n)$, where $f(\cdot)$ is a concrete function. There are many estimates on $w(P)$ in the papers \cite{BLPS99,BAN96,DAD,RUD00}. There is a conjecture claiming that $f(n) = O(n)$. The best known upper bound on $f(n)$ is $O(n^{4/3}\, log^c(n))$ due to Rudelson \cite{RUD00}, where $c$ is some constant that does not depend on $n$.

An interesting problem is to estimate the width of empty lattice simplices \cite{BLPS99,HAAZ00,KANT97,SEB99}.

\begin{definition}
A simplex $S$ is called \emph{empty lattice} if $vert(S) \subseteq \mathbb{Z}^n$ and $(S \cap \mathbb{Z}^n )\setminus vert(S) = \emptyset$.
\end{definition}

The best known estimate for empty lattice simplices is $f(n)=O(n\,log(n))$ \cite{BLPS99}.

In our previous papers \cite{GRIB13,GRIB16}, we have proved the following analogues of the flatness theorem for polytopes.

\begin{theorem}
Let $A \in \mathbb{Z}^{m \times n}$, $b \in \mathbb{Z}^m$, and $P = P(A,b)$ be a polytope. If $w(P) > (n+1)\Delta(A) \cfrac{\Delta_{lcm}(A)-1}{\Delta_{gcd}(A)}$, then $|P \cap \mathbb{Z}^n| \geq n+1$. Moreover, there is a polynomial-time algorithm to find an element of $P \cap \mathbb{Z}^n$.
\end{theorem}

\begin{theorem}\label{SimplexFlatTh}
Let $A \in \mathbb{Z}^{(n+1) \times n}$, $b \in \mathbb{Z}^{n+1}$, and $P = P(A,b)$ be a simplex. If $w(P) > \delta(A)$, then $P \cap \mathbb{Z}^n \not= \emptyset$.
Moreover, there is a polynomial-time algorithm to find an element of $P \cap \mathbb{Z}^n$.
\end{theorem}

In other words, we have shown that if a polyhedron is broad enough, then there is
its integer point and some its integer point can be found in polynomial time.
The estimate of the first theorem can be excessive for matrices with large $\Delta_{lcm}(A)$, but it can be better for some special classes of matrices, for example $k$-modular matrices \cite{KOT14}. In this paper, we also deal with minors
of the constraint matrices.

Seb\"{o} shown \cite{SEB99} that the problem of computing the width of simplices defined by systems of rational inequalities is NP-hard. In this paper, we show that this problem is polynomial in the case, when values of some minors of the constraint matrices are bounded.
We also show solvability of the integer programming problem in quasi-polynomial or polynomial time on "second-order simplices" defined by taking simplices minus their integer vertices assuming that some minors of the constraint matrices of the simplices are bounded.

The authors consider this paper as a part for achieving the general aim to find out critical values of parameters, when a given problem changes complexity. For example, the integer linear programming problem is polynomial-time solvable on polyhedrons with all-integer vertices, due to \cite{KHA80}, see also \cite{HGLOB95}. On the other hand, it is NP-complete in the class of polyhedrons with denominators of extreme points equal $1$ or $2$, due \cite{PAD89}. The famous $k$-satisfiability problem is polynomial for $k \leq 2$, but is NP-complete for all $k > 2$. A theory, when an NP-complete graph problem becomes easier, is investigated for the family of hereditary classes in the papers \cite{A,AKL,ABKL,M1,M2,KLMT,M3,M4,MP}.

\section{Unimodular cone decomposition}

A proof of the following theorem can be found in \cite{SHEV96}, and the very similar fact can be found in \cite{BARV96,BARVP99}.
\begin{theorem}
Let $A \in \mathbb{R}^{m \times n}$, $b \in \mathbb{R}^m$, and $b = A y$ for some $y \in \mathbb{R}^n$. Let $J_{+} = \{j :\: x_j > 0\}$ and $J_{-} = \{j :\: x_j < 0\}$. Then $cone(A) = \left(\bigcup\limits_{j \in J_{-}} cone(A[j,b])\right) \cap \left(\bigcup\limits_{j \in J_{+}} cone(A[j,b])\right)$, where $A[j,b]$ is the matrix obtained from $A$ by replacing $j$-th column with the column $b$.
\end{theorem}

\begin{definition}
A set of cones $\{C_1,C_2,\dots,C_s\}$ is called a \emph{decomposition of a cone} $C$, if $\bigcup\limits_{i = 1}^s C_i = C$. A decomposition is called \emph{strict}, if $dim(C_i \cap C_j) < dim(C)$ for any $i \not= j$.
\end{definition}

\begin{corollary}\label{SplitCor}
Let $A \in \mathbb{R}^{m \times n}$, $b \in \mathbb{R}^m$, and $b = A y$ for some $y \in \mathbb{R}^n_{+}$. Then $cone(A) = \bigcup\limits_{j \in J_+} cone(A[j,b])$ is a strict decomposition of $cone(A)$, where $A[j,b]$ is the matrix obtained from $A$ by replacing $j$-th column with the column $b$.
\end{corollary}

\begin{lemma}\label{SplitLem}
For any $k$, the equality $det(A[k,b]) = y_k det(A)$ holds.
\end{lemma}

\begin{proof}
$det(A[k,b]) = det(A[k,Ay]) = det(A[k, y_k A_{*\,k}]) = y_k det(A)$. \qed
\end{proof}

Let $A \in \mathbb{Z}^{n \times n}$, $|det(A)| = \Delta > 0$. Let $A = PSQ$, where $S$ is the Smith normal form \cite{SCHR98} of $A$ and $P,Q$ are unimodular matrices. Let us consider the system:
\begin{equation}\label{ModSystem}
\begin{cases}
A x \equiv 0\,(mod\, S_{n\,n}) \\
0 \leq x \leq S_{n\,n} - 1 \\
x \in \mathbb{Z}^n \\
\end{cases}.
\end{equation}

All solutions of this system can be found with the following formula:

\begin{equation}\label{ModSystemSol}
x = Q^{-1} S^{-1} S_{n\,n} t \mod S_{n\,n}, \text{ where } t \in \mathbb{Z}^n.
\end{equation}

We can use solutions of the system (\ref{ModSystem}) to make a strict unimodular decomposition of $cone(A)$. To do that, we need to take some non-zero solution $x$ of the system (\ref{ModSystem}) and put $b = A \cfrac{x}{S_{n\,n}}$. Using Corollary \ref{SplitCor}, we can obtain a strict decomposition $cone(A) = \bigcup\limits_{j \in J_+} cone(A[j,b])$. By Lemma \ref{SplitLem}, we have that $|det(A[j,b])| = \cfrac{x_j}{S_{n\,n}} \Delta < \Delta$. Therefore, we may use a recursion
based on decreasing the determinants of appearing matrices. The total number of cones in the decomposition will be at most $n^{\Delta}$. 

The very efficient decomposition algorithm called \emph{signed decomposition} can be found in works \cite{BARV96,BARVP99}, but for our purposes we need decomposition with only positive signs.

\medskip We will follow an algorithm that was proposed by A. Y. Chirkov. The algorithm computes a non-strict unimodular decomposition of $cone(A)$ on at most $n^{2 \log_2(\Delta)}$ cones. The algorithm was mentioned in \cite{SHEV96}, but it was never published.

\begin{theorem}\label{NotStrictDecomp}
Let $A,B,C \in \mathbb{R}^{n \times n}$, $|det(A)| = \Delta > 0$, and $B = A C$, where $0 \leq C_{i\,j} \leq C_{i\,i}$ for any $i,j \in \{1,\dots,n\}$. Let $I = \{ i : C_{i\,i} > 0 \}$. If $I \not= \emptyset$, then $cone(A) = \bigcup\limits_{i \in I} cone(A[i,B_{*\,i}])$.
\end{theorem}

\begin{proof}
By Corollary \ref{SplitCor}, we have $cone(A[j,B_{*\,j}]) \subseteq cone(A)$. Suppose that $x \in cone(A)$. Then $x = A t$ for some $t \in \mathbb{Z}_+^n$. Let $t_k/C_{k\,k} = \max \{t_j/C_{j\,j} :\: C_{j\,j} \not= 0,\, j \in \{1,\dots,n\} \}$. Then $x = t_k/C_{k\,k} B_{*\,k} + \sum\limits_{j \not= k} (t_j - t_k C_{j\,k}/C_{k\,k}) A_{*\,j}$. Since $(t_j - t_k C_{j\,k}/C_{k\,k}) \geq (t_j - t_k C_{j\,j}/C_{k\,k}) \geq 0$, we have that $x \in cone(A[k,B_{*\,k}])$ is true.\qed
\end{proof}

\begin{theorem}\label{UnimodDecomp}
Let $A \in \mathbb{Z}^{n\times n}$, $|det(A)| = \Delta > 0$. There is an algorithm to compute integer unimodular matrices $B^{(j)}$, where $j \in \{1,\dots,s\}$, such that $cone(A) = \bigcup\limits_{j=1}^s cone(B^{(j)})$. The value $s \leq n^{2 \log_2(\Delta)}$ is the total number of cones in the decomposition of $cone(A)$. The algorithm is polynomial for a fixed $\Delta$.
\end{theorem}

\begin{proof} Firstly, suppose that $\Delta$ is even. Let $q$ be a solution of the system \eqref{ModSystem}, such that at list one of the components of $q$ is odd. For example, using the formula \eqref{ModSystemSol}, we may put $q = (Q^{-1})_{*\,n}$. The column $(Q^{-1})_{*\,n}$ must have an odd component, because it is a column of the unimodular matrix $Q^{-1}$.

Let us consider the vector $q^\prime = \cfrac{S_{n\,n}}{2}\, q \mod S_{n\,n}$. If a component $q_i$ is even, then $q^\prime_i = 0$. If $q_i$ is odd, then $q^\prime_i = S_{n\,n}/2$. Hence, $q^\prime$ is the non-zero solution of the system \eqref{ModSystem} with the property $0 \leq q^\prime \leq S_{n\,n}/2$. So, we can use Corollary \ref{SplitCor} with the vector $b = A\cfrac{q^\prime}{S_{n\,n}}$ to obtain a decomposition $cone(A) = \bigcup\limits_{j \in J_+} cone(A[j,b])$. By the Lemma \ref{SplitLem}, if $q^\prime_j = 0$, then $|det(A[j,b])| = 0$ and if $q^\prime = S_{n\,n}/2$, then $|det(A[j,b])| = \Delta / 2$.

Suppose now that $\Delta$ is odd. Let $x$ be some non-zero solution of the system \eqref{ModSystem}. Let us consider the vectors $y^{(k)} \in \mathbb{Z}_+^n$ for $k \in \{1,\dots,n\}$. Let the components of $y^{(k)}$ be $y^{(k)}_i = \lambda_k x_i \mod S_{n\,n}$, where $\lambda_k$ is the solution of the equation $\lambda_k x_k \equiv - gcd(S_{n\,n},x_k)\: (mod\, S_{n\,n})$. Recall that $gcd(a,b)$ denotes the greatest common divisor of two naturals $a$ and $b$. Let us consider the matrix $C \in \mathbb{Q}^{n \times n}$, such that $C_{*\,k} = y^{(k)}/\Delta$. By the definition of $y_i^{(k)}$, we have $gcd(S_{n\,n},x_i) \mid y_i^{(k)}$. The maximal number $y_i^{(k)}$ of that type is $y_i^{(i)} = S_{n\,n} - gcd(S_{n\,n},x_i)$. Since $C_{i\,k} = y_i^{(k)}/\Delta$, we have $0 \leq C_{i\,k} \leq C_{i\,i}$ for any $i,k \in \{1,\dots,n\}$. Moreover, $|det(A[k,B_{*\,k}])| = y^{(k)}_k = S_{n\,n} - gcd(S_{n\,n},x_k)$, which is an even number. Therefore, we can use Theorem \ref{NotStrictDecomp} with the matrix $B = A C$ to obtain the decomposition $cone(A) = \bigcup\limits_{i \in I} cone(A[i,B_{*\,i}])$, where $I = \{i : C_{i\,i} > 0\}$. After this step all determinants of cones becomes even numbers, and we can make the decomposition step for the even case.

It is easy to see that the total number of steps is at most $2 \log_2(\Delta)$. Hence, the total number of unimodular cones in the decomposition is at most $n^{2 \log_2(\Delta)}$. Moreover, all subroutines of this algorithm are polynomial-time.\qed
\end{proof}

\begin{theorem}
Let $cone(A) = \bigcup\limits_{i = 1}^{s} cone(B^{(j)})$ be a decomposition of $cone(A)$ that was obtained after $k \geq 1$ steps of the algorithm. Let $b$ be some column of the matrix $B^{(j)}$ for some $j$ and $b = A t$ for some $t \in \mathbb{Q}^n_+$. Then $|t|_{\infty} < 2^{k-1}$.
\end{theorem}

\begin{proof}
Let us start the induction from $k = 1$ and consider the vector $t$ after the first step of the algorithm. If $b$ is a column of the initial matrix $A$, then  $|t|_{\infty} = 1$. If $b$ is a column that was added by the algorithm at the first step, then we have $|t|_{\infty} = 1/2$ if $\Delta$ is even and $|t|_{\infty} < 1$ if $\Delta$ is odd.

Let us consider the general case. Let $cone(A) = \bigcup\limits_{j=1}^m cone(A[j,a^{(j)}])$ be a decomposition after the first step of the algorithm, where $a^{(j)} \in \mathbb{Z}^n$ and $1 \leq m \leq n$. Without loss of generality, we can assume that $b \in cone(A[1,a^{(1)}])$. So, the column $b$ is obtained by applying the algorithm to $cone(A[1,a^{(1)}])$ using $k-1$ steps. Hence, there exists $y \in \mathbb{Q}^n_+$ such that $b = y_1 a^{(1)} + \sum\limits_{j=2}^n y_j A_{*\,j}$. We know that $a^{(1)} = A z$ for some $z \in \mathbb{Q}^n_+$ and $|z|_\infty < 1$. Hence, $b = At = z_1 y_1 A_{*\,1} + \sum\limits_{j=2}^n (z_j y_1 + y_j) A_{*\,j}$. By induction, we have $|y|_\infty < 2^{k-2}$.

Finaly, $t_j =
\begin{cases}
z_1 y_1 < y_1 < 2^{k-2}, \text{ for } j = 1 \\
z_j y_1 + y_j < y_1 + y_j < 2^{k-1}, \text{ for any } j \in \{2,\dots,n\}. \\
\end{cases}$\qed
\end{proof}

\begin{corollary}
Let $\{B^{(j)}:~j \in \{1,2,\dots,s\}\}$ be a unimodular decomposition from Theorem \ref{UnimodDecomp}. Let $b$ be the column of some $B^{(j)}$ and $b = A t$ for some $t \in \mathbb{Q}^n_+$. Then $|t|_{\infty} \leq \Delta^2$.
\end{corollary}

\begin{proof}
The corollary follows from Theorem 6 and the fact that the number of steps of the algorithm is at most $2\,log_2(\Delta)$.\qed
\end{proof}

\begin{theorem}\label{DecompositionComplexity}
The bit complexity of the algorithm from Theorem \ref{UnimodDecomp} is

$O(n^{\Theta + 2\,log_2(\Delta)} M(n\, log(n \alpha \Delta^2) ) )$, where $\Theta$ is the matrix multiplication exponent, $\Delta = |det(A)|$, $\alpha = \max\{|A_{i\,j}|\}$, and $M(t)$ is the complexity of multiplication of two $t$-bits integers.
\end{theorem}

\begin{proof}
The most complex part of the algorithm during $k$-th step is computing of many Smith normal forms. Due to Storjohann \cite{STOR96,ZHEN05}, the Smith normal form computation complexity for $n \times n$ matrix $A$ is $O(n^{\Theta} M(n log\, \alpha) )$. Let $cone(A) = \bigcup\limits_{i=1}^{s} cone(B^{(i)})$ be a decomposition that was obtained after $k$ steps of the algorithm. We know that $s \leq n^{k}$ and $B^{(k)} = A T^{(k)}$, such that $0 \leq T^{(k)}_{i\,j} \leq 2^{k-1}$. Hence, $|B^{(k)}|_\infty \leq n \alpha 2^{k-1}$. So, the total cost of $k$-th step is $O(n^{k-1} n^\Theta M(n\, log(n \alpha 2^{k-1})))$. Finally, the total cost of the algorithm is $\sum\limits_{k=1}^{s} O(n^{k-1} n^\Theta M(n\, log(n \alpha 2^{k-1}))) = O(n^{\Theta + 2\,log_2(\Delta)} M(n\, log(n \alpha \Delta^2) ))$.\qed 
\end{proof}

\section{Computing the width of simplices}

\begin{lemma}\label{UnimodSectionLm}
Let $p \in \mathbb{Z}^n$, $B \in \mathbb{Z}^{n \times n}$, and $det(B) \not= 0$. Let $A \in \mathbb{Z}^{m \times n}$, $b \in \mathbb{Z}^n$, and for any $i \in \{1,\dots,m\}$ we have ${(A_{i\,*})}^\top \in cone({(B^{-1}})^\top)$. Then $P(A,b) \cap (p + cone(B)) \cap \mathbb{Z}^n \not= \emptyset$ if and only if $p \in P(A,b)$.
\end{lemma}

\begin{proof}
If $p \notin P(A,b)$, then there exists $i \in \{1,\dots,m\}$, such that $A_{i\,*} p > b_i$. For any $x \in (p + cone(B)) \cap \mathbb{Z}^n$, we have $A_{i\,*} x = A_{i\,*} p + A_{i\,*} c$, where $c \in cone(B)$. Since ${(A_{i\,*})}^\top \in cone({(B^{-1})}^\top)$, we have $A_{i\,*} = t^\top B^{-1}$ for some $t \in \mathbb{Q}^n_+$. Since $A_{i\,*} p > b_i$ and $B^{-1} c > 0$, we have $A_{i\,*} x > b_i$. \qed
\end{proof}

\begin{lemma}\label{PolyLm}
Let $C \in \mathbb{Z}^{n \times n}$, $p \in \mathbb{Q}^n$, $|det(C)| = \Delta > 0$. Let $A \in \mathbb{Z}^{m \times n}$, $b \in \mathbb{Z}^n$, $c \in \mathbb{Z}^n$, and for any $i \in \{1,\dots,m\}$ we have $({A_{i\,*})}^\top \in cone({(C^{-1})}^\top)$ and $c \in cone({(C^{-1})}^\top)$.
Then, for every fixed $\Delta$, there is a polynomial-time algorithm to solve the problem $\max \{c^\top x : x \in P(A,b) \cap (p + cone(C)) \cap \mathbb{Z}^n \}$. The algorithm's complexity is the same as the algorithm's complexity in Theorem \ref{DecompositionComplexity}.
\end{lemma}

\begin{proof}
Let, by Theorem \ref{UnimodDecomp}, $\bigcup\limits_{i = 1}^s cone(B^{(i)})$ be a unimodular decomposition of $cone(C)$, where $B^{(i)} \in \mathbb{Z}^{n \times n}$ are unimodular matrices and $s \leq n^{2 \log_2(\Delta)}$. Hence, $cone(B^{(i)}) = P(-{(B^{(i)})}^{-1}, 0)$ and $p + cone(C) = \bigcup\limits_{i = 1}^s P(-{(B^{(i)})}^{-1}, {-(B^{(i)})}^{-1} p)$. Moreover, $(p + cone(C)) \cap \mathbb{Z}^n = \bigcup\limits_{i = 1}^s P(-{(B^{(i)})}^{-1}, -\sigma^{(i)})$, where $\sigma^{(i)} = \lfloor {-(B^{(i)})}^{-1} p \rfloor$. Let $x^{(i)} = B^{(i)} \sigma^{(i)} \in \mathbb{Z}^n$. Then, we have $(p + cone(C)) \cap \mathbb{Z}^n = \bigcup\limits_{i = 1}^s (x^{(i)} + cone(B^{(i)}) )$.

Finally, we need to determine the feasibility of the sets $P(A,b) \cap (x^{(i)} + cone(B^{(i)})) \cap \mathbb{Z}^n$ for each $i \in \{1,\dots,s\}$. By Lemma \ref{UnimodSectionLm}, it can be done in polynomial time by checking whether $x^{(i)}$ belongs to $P(A,b)$. If $P(A,b) \cap (x^{(i)} + cone(B^{(i)})) \cap \mathbb{Z}^n$ is feasible, then $c^\top x^{(i)} = \max \{ c^\top x : P(A,b) \cap (x^{(i)} + cone(B^{(i)})) \cap \mathbb{Z}^n \}$, since $c \in cone({(C^{-1})}^\top)$. If every set is empty, then $P(A,b) \cap (p + cone(C)) \cap \mathbb{Z}^n$ is empty too.

So, our algorithm contains the unimodular decomposition step and $s$ checking steps. It is easy to see that the complexity of the algorithm
is the same as the complexity of the unimodular decomposition algorithm (see Theorem \ref{UnimodDecomp}). \qed
\end{proof}

\begin{corollary}\label{PolyCor}
Let $C \in \mathbb{Z}^{n \times n}$, $p \in \mathbb{Q}^n$, $|det(C)| = \Delta > 0$. Let $A \in \mathbb{Z}^{m \times n}$, $b \in \mathbb{Z}^n$, and for any $i \in \{1,\dots,m\}$ we have ${(A_{i\,*})}^\top \in cone({(C^{-1})}^\top)$. Then, if $\Delta$ is fixed, there is a polynomial-time algorithm for the feasibility problem in the set $P(A,b) \cap (p + cone(C)) \cap \mathbb{Z}^n$.
\end{corollary}

\begin{theorem}
Let $A \in \mathbb{Z}^{(n+1) \times n}$, $b \in \mathbb{Z}^{n+1}$, and $P=P(A,b)$ be a $n$-dimensional simplex. If $\max\{\Delta_n(A\,b),\Delta_{n-1}(A)\}$ is fixed, then there is a polynomial-time algorithm to find the width and the flat direction of $P$. The algorithm's complexity is $O(T(n,\Delta_{n-1}(A)) n^3 \Delta(A\,b) \Delta(A) )=\tilde O(n^{3 + \Theta + 2 log_2\, \Delta_{n-1}(A)})$, where $\Theta$ is the matrix multiplication exponent and $T(n,r)$ denotes the complexity of the algorithm of Lemma 5 for square integral $n-1\times n-1$ matrices with the greatest common divisor of the determinants of all rank-submatrices, equal to r.
\end{theorem}

\begin{proof}
Since $P$ is a simplex, $\mathbb{R}^n = \bigcup\limits_{v \in vert(P)} N(v)$. Hence, $\mathbb{R}^n = \bigcup\limits_{v,u \in vert(P)} N(v) \cap (-N(u))$. Let $M(v,u) = N(v) \cap (-N(u)) \cap \mathbb{Z}^n \setminus \{0\}$. Using the previous equation, we have $width(P) =\min\limits_{v,u \in vert(P)} \min\limits_{c \in M(v,u)}
(\max\limits_{x \in P} c^\top x - \min\limits_{x \in P} c^\top x) =\min\limits_{v,u \in vert(P)} \min\limits_{c \in M(v,u)} c^\top (v-u)$. Hence, the problem to find the width and the flat direction of $P$ is equivalent to a family of $O(n^2)$ problems generated by all pairs of vertices of $P$.

Let us consider two different vertices $v,u$ of $P$ and consider the problem $\min\limits_{c \in M(v,u)}(c^\top(v - u))$.

Since $v,u$ are adjacent to the edge $v - u$, the vertices $v,u$ have $n-1$ common facets. Let a matrix $B^\top$ be induced by rows of the matrix $A$ that correspond to $n-1$ common facets. Hence, we can assume that $N(u) = cone(B \, a_u)$ and $N(u) = cone(B \, a_v)$, where $a_u^\top,a_v^\top$ are rows of the matrix $A$. Additionally, we have $-a_u \in N(v)$, $-a_v \in N(u)$ and, finally, $a_v,-a_u \in M(v,u)$.

Let us consider the hyperplane $H(k) = \{ x \in \mathbb{R}^n : (v - u) x = k \}$. Since $\forall c \in M(v,u)$ we have $ c^\top (v-u) \geq 0$, the equality  $M(v,u) = \bigcup\limits_{k \in \mathbb{Z}_+} (M(v,u) \cap H(k))$ is true. Hence, $\min\limits_{c \in M(u,v)}c^\top(u - v)=\min \{ k \in \{1,2,\dots,s\} : M(v,u) \cap H(k) \not= \emptyset \}$, where $s = \min \{a_v^\top (v-u), -a_u^\top(v-u)\}$.   The next lemma will help us to check the emptiness of the set $M(v,u) \cap H(k)$.

\begin{lemma}
Let $k \in \mathbb{R}_+$, then $N(v) \cap (-N(u)) \cap H(k) = (p_v(k) + cone(B)) \cap (p_u(k)-cone(B))$, where $p_v(k)$ is the intersection point of the ray $L_v = \{a_v t : t \in \mathbb{R}_+ \}$ with $H(k)$ and $p_u(k)$ is the intersection point of the ray $L_u = \{-a_u t : t \in \mathbb{R}_+ \}$ with $H(k)$.
\end{lemma}

\begin{proof}
Let $x \in N(v) \cap (-N(u)) \cap H(k)$, then $x = B \alpha + a_v t_v = -B \beta - a_u t_u$ for some $\alpha,\beta \in \mathbb{Q}_+^{n-1}$ and $t_v,t_u \in \mathbb{Q}_+$. Additionally, we have $x^\top (v-u) = k$. Since $B^\top v = B^\top u$, we have $t_v = k/a_v^\top (v-u)$ and $t_u = -k/a_u^\top (v-u)$. Let us consider the points $p_v(k)$ and $p_u(k)$. It is easy to see that $p_v(k) = a_v t_v$ and $p_u(k) = -a_u t_u$. Since $x = B \alpha + a_v t_v = -B \beta - a_u t_u$, we have that $x \in (p_v(k) + cone(B)) \cap (p_u(k)-cone(B))$.

Let $x \in (p_v(k) + cone(B)) \cap (p_u(k)-cone(B))$. Since $B^\top(v-u) = 0$, we have $x^\top (v-u) = p_v(k)^\top (v-u) = k$, and so $x \in H(k)$. Finally, $x \in N(v)$ and $x \in -N(v)$, because the points $p_v(k)$, $p_u(k)$ lie on the rays $L_v$, $L_u$. These rays are generating rays for the cones $N(v)$ and $-N(u)$. It finishes the proof of this lemma. \qed
\end{proof}

\begin{lemma}\label{LayerProbLm}
Let $k \in \mathbb{N}$. The emptiness of the set $M(v,u) \cap H(k)$ can be checked with a polynomial-time algorithm, if $\Delta_{gcd}(B)$ is fixed.
\end{lemma}

\begin{proof}
Let $d$ be the greatest common divisor of the components of $v-u$. So, we can find a unimodular matrix $Q \in \mathbb{Z}^{n \times n}$, such that $(v-u)Q = (d,0,\dots,0)$ and, hence, $Q_{1\,*} = (v-u)/d$. After the unimodular map $x \to Q x$ the cone $(p_v(k) + cone(B)) \cap (p_u(k)-cone(B))$ becomes the cone $(Q^{-1}p_v(k) + cone(Q^{-1}B)) \cap (Q^{-1}p_u(k)-cone(Q^{-1}B))$. Since $(v-u)p_v(k)/d = (v-u)p_u(k)/d = k/d$ and $(v-u)B = 0$, we have $Q^{-1} B = \dbinom{0}{B^\prime}$, $Q^{-1} p_v(k) = \dbinom{k/d}{p_v^\prime}$, and $Q^{-1} p_u(k) = \dbinom{k/d}{p_u^\prime}$ for some $B^\prime \in \mathbb{Z}^{(n-1) \times (n-1)}$, $p_v^\prime, p_u^\prime \in \mathbb{Z}^{n-1}$. Hence, $M(v,u) \cap H(k)$ is not empty if and only if the set $(p_v^\prime + cone(B^\prime)) \cap (p_u^\prime - cone(B^\prime))$ has at least one integer point and $d \mid k$. Let a system $R x \leq r$ be the dual representation of the cone $p_u^\prime - cone(B^\prime)$. In other words, $P(R,r) = p_u^\prime - cone(B^\prime)$. We can take $R = {(B^\prime)}^{-1}$ and $r = R p_u^\prime$ and after that make this matrices integral by multiplying by an appropriate integer value that has a polynomial size. Finally, our problem is the feasibility problem in the set $p_v^\prime + cone(B^\prime) \cap P(R,R) \cap \mathbb{Z}^{n-1}$. It is easy to see that this polytope satisfies the conditions of Corollary \ref{PolyCor}, moreover $\Delta(B^\prime) = \Delta_{gcd}(n-1, Q^{-1} B) = \Delta_{gcd}(n-1,B)$. So, the feasibility problem in the set $p_v^\prime + cone(B^\prime) \cap P(R,r) \cap \mathbb{Z}^{n-1}$ and the initial problem can be solved in polynomial time if $\Delta_{gcd}(n-1,B)$ is fixed. It finishes the proof of this lemma. \qed
\end{proof}

Let us estimate the complexity of the algorithm. For every two vertices $v,u$ of the simplex, the algorithm solves the problem $\min\limits_{c \in M(v,u)}c^\top(v - u)$. Every problem can be solved using the equality $\min\limits_{c \in M(v,u)}c^\top(v - u)= \min \{ k \in \{1,2,\dots,s\} : M(v,u) \cap H(k) \not= \emptyset \}$, where $s = \min \{a_v^\top (v-u), -a_u^\top(v-u)\}$. Every feasibility subproblem in the set $M(v,u) \cap H(k)$ can be solved using Lemma \ref{LayerProbLm}. Hence, the total complexity is $O(T(n,r) n^2 s)$, where $r = \max\{\Delta_{gcd}(n-1,B) : B \text{ is } (n-1) \times n \text{ submatrix of } A\}$. Let us estimate $s = \min \{a_v^\top (v-u), -a_u^\top(v-u)\}$. Due to Cramer's rule, for every vertex $x$ of the simplex we have $||x||_{\infty} \leq \Delta(A\,b) / \delta(A)$. Additionally, we have $||a_v||_\infty \leq \Delta_1(A)$ and $||-a_u||_\infty \leq \Delta_1(A)$. Hence, $s \leq 2 n \Delta(A\,b) \Delta_1(A)$. Totally, the complexity is $O(T(n,r) n^3 \Delta(A\,b) \Delta_1(A))$. Next, we can eliminate the $\Delta_1(A)$ multiplier. To do that, we can assume that the matrix $A$ of the simplex $P$ has already been transformed to the Hermite normal form \cite{SCHR98}. In other words, we can assume that
$A = \begin{pmatrix}
A_{1\,1} & 0 & 0 & \dots & 0 & 0 \\
A_{2\,1} & A_{2\,2} & 0 & \dots & 0 & 0 \\
\hdotsfor{6} \\
A_{n\,1} & A_{n\,2} & A_{n\,3} & \dots & A_{n\,n-1} & A_{n\,n} \\
A_{n+1\,1} & A_{n+1\,2} & A_{n+1\,3} & \dots & A_{n+1\,n-1} & A_{n+1\,n} \\
\end{pmatrix}$, where $0 \leq A_{i\,j} < A_{i\,i} < \Delta(A)$ for $1 \leq i,j \leq n$. Notice that these inequalities are false for the last row of $A$. Hence, one of the inequalities $||a_v||_\infty \leq \Delta(A)$ or $||-a_u||_\infty \leq \Delta(A)$ is true. Hence, we have that $s \leq 2 n \Delta(A\,b) \Delta(A)$. We also note that $r = \max\{\Delta_{gcd}(n-1,B) : B \text{ is } (n-1) \times n \text{ submatrix of } A\}$ is the invariant under any unimodular map that transforms the initial system to a system in the Hermite normal form. Since $ r \leq \Delta_{n-1}(A)$, we have the final formula $O(T(n,\Delta_{n-1}(A)) n^3 \Delta(A\,b) \Delta(A) )$ for the complexity. It finishes the proof
of the theorem. \qed

\end{proof}

Additionally, if for the simplex $P(A,b)$ we have $P(A,b) \cap \mathbb{Z}^n = \emptyset$, then we can use Theorem \ref{SimplexFlatTh} to bound its width. In this case, the algorithm's complexity can be reduced.

\begin{theorem}
Let $A \in \mathbb{Z}^{(n+1) \times n}$, $b \in \mathbb{Z}^{n+1}$, $P=P(A,b)$ be a $n$--dimensional simplex and $P \cap \mathbb{Z}^n = \emptyset$. Then, if $\max\{\delta_n(A),\Delta_{n-1}(A)\}$ is fixed, there is a polynomial-time algorithm to find the width and the flat direction of $P$. The algorithm's complexity is $O(T(n,\Delta_{n-1}(A)) n^2 \delta(A) )$ or $\tilde O(n^{3 + \Theta + 2 log_2\, \Delta_{n-1}(A)})$, where $\Theta$ is the matrix multiplication exponent.
\end{theorem}

\begin{proof}
From the proof of the previous theorem we have that $\min\limits_{c \in M(v,u)}c^\top(v - u)= \min \{ k \in \{1,2,\dots,s\} : M(v,u) \cap H(k) \not= \emptyset \}$, where $s = \min \{a_v^\top (v-u), -a_u^\top(v-u)\}$. But, it is not necessary to enumerate all $k \in \{1,2,\dots,s\}$, since the width of $P$ is bounded. Since $P \cap \mathbb{Z}^n = \emptyset$, we can use Theorem \ref{SimplexFlatTh} and bound the width by $\delta_n(A) = \delta(A)$. Hence,
$\min\limits_{c \in M(v,u)}c^\top(v - u)= \min \{ k \in \{1,2,\dots,\delta_n(A)\} : M(v,u) \cap H(k) \not= \emptyset\}$.
Hence, we need to solve only $\delta_n(A)$ subproblems.  The remaining part of the proof is the same as in the previous theorem.\qed
\end{proof}

\section{On the complexity of the integer optimization in a simplex}

Let $P = P(A,b)$ be a simplex, where $A \in \mathbb{Z}^{(n+1) \times n}$ and $b \in \mathbb{Z}^{n+1}$. Due to Gomory's papers \cite{GOM65}, see also \cite{HU70}, the feasibility problem in the set $P \cap \mathbb{Z}^n$ can be reduced to the group minimization problem and can be solved in polynomial time if $\Delta(A)$ is fixed. Using Papadimitriou's  dynamic programming approach \cite{PAPA}, the problem $\max\{c^\top x : x \in P \cap \mathbb{Z}^n\}$ can also be solved in polynomial time if $\Delta_1(A \,b)$ and $\Delta(A)$ are fixed. To see this, we need to reduce the initial system $Ax\leq b$ with $n+1$ inequalities to a system $\begin{cases}
B y + C z\leq b^\prime \\
y_i \geq 0 \\
\end{cases}$, where the system $B y + C z\leq b^\prime$ has only $O(log_2 \, \Delta(A))$ inequalities. This reduction step can be done using the Hermite normal form \cite{SCHR98}.

In this work, we consider simplices that are generated by convex hulls of integer points and have bounded sub-determinants of a restrictions matrix. We note that the feasibility problem in such simplicies can be solved in polynomial time due to the integer points counting algorithm from \cite{BARV96,BARVP99}.

\begin{theorem}
Let $A \in \mathbb{Z}^{n \times (n+1)}$, $c \in \mathbb{Z}^n$, and $S = conv(A)$ be a $n$-dimensional simplex. If $\Delta(A)$ is fixed, then there is an algorithm quasi-polynomial on $n$ to solve the problem $\max\{c^\top x : x \in S\cap \mathbb{Z}^n \setminus vert(S)\}$. Additionally, let $B_i = \left( A_{*\,1} - A_{*\,i}, A_{*\,2} - A_{*\,i},\dots, A_{*\,n+1} - A_{*\,i} \right)$. Let  $\alpha = \max\limits_i\{|det(B_i)|\}$. If $\alpha$ is fixed, then there is a polynomial-time algorithm for the problem $\max\{c^\top x : x \in S\cap \mathbb{Z}^n \setminus vert(S)\}$.
\end{theorem}

\begin{proof}
Let us assume that the optimum of the linear problem $\max\{c^\top x : x \in S \}$ is reached on the vector $A_{*\,1}$. 

Consider the matrix $B_{1} = \left(A_{*\,2} - A_{*\,1}, A_{*\,3} - A_{*\,1},  \,\dots, A_{*\,n+1} - A_{*\,1} \right)$.
By definition, the columns of the matrix $B_1$ are edges of the simplex that are adjacent to the vertex $A_{*\,1}$.
There exists a halfspace $H = \{x \in \mathbb{R}^n : a^\top x \leq a_0 \}$, such that $P = (A_{*\,1} + cone(B_{1})) \cap H$. The inequality $a^\top x \leq a_0$ can be trivially found in polynomial time as the inequality corresponding to the facet opposite to the vertex $A_{*\,1}$. Since $S$ is a simplex, we have $a \in cone({(B_1^{-1})}^\top)$ and $c \in cone({(B_1^{-1})}^\top)$. Finally, $|det(B_1)| \leq (n+1) \Delta(A)$ and we can use the algorithm from Lemma \ref{PolyLm}. The complexity of the algorithm is $\tilde O(n^{\Theta + 2 log_2\,n + 2log_2 \,\Delta(A)})$. This gives us a quasi-polynomial algorithm's complexity bound.

Consider the case, when $\alpha$ is fixed. We know that $|det(B_1)| \leq \alpha$. By the same argument, the algorithm's complexity becomes equal to $\tilde O(n^{\Theta + 2 log_2 \,\alpha})$, i.e. becomes polynomial on $n$.\qed
\end{proof}

\section*{Acknowledgments}

The first author would like to thank Prof. P.M. Pardalos and his academic supervisor Prof. D.S. Malyshev.

\end{document}